\documentclass[10pt]{amsart}

\usepackage{tikz}
\usepackage{graphicx}
\usepackage{tikz-cd}
\usepackage[shortlabels]{enumitem}

\usetikzlibrary{decorations.pathmorphing}
\tikzset{snake it/.style={decorate, decoration=snake}}

\newtheorem{theorem}{Theorem}[section]
\newtheorem{lemma}[theorem]{Lemma}
\newtheorem{proposition}[theorem]{Proposition}
\newtheorem{corollary}[theorem]{Corollary}
\newtheorem{conjecture}[theorem]{Conjecture}

\theoremstyle{definition}
\newtheorem{definition}[theorem]{Definition}
\newtheorem{example}[theorem]{Example}

\theoremstyle{remark}
\newtheorem{remark}[theorem]{Remark}

\numberwithin{equation}{section}

\newcommand{\abs}[1]{\lvert#1\rvert}

\begin{document}

\title{Heavy sets and index bounded relative symplectic cohomology}

\author{Yuhan Sun}

\date{\today}

\address{Hill Center, Department of Mathematics, Rutgers University, 110 Frelinghuysen Rd, Piscataway NJ 08854.}

\email{sun.yuhan@hotmail.com}

\begin{abstract}
We use relative symplectic cohomology to detect heavy sets, with the help of index bounded contact forms. This establishes a relation between two notions SH-heaviness and heaviness, which partly answers a conjecture of Dickstein-Ganor-Polterovich-Zapolsky in the symplectically aspherical setting. 
\end{abstract}

\maketitle

\tableofcontents


\section{Introduction}

Heavy and super-heavy sets were introduced by Entov and Polterovich in \cite{EP}. These notions express certain aspects of symplectic rigidity of subsets in closed symplectic manifolds, through deep relations between the quantum cohomology and the Hamiltonian dynamics. In this article we will use a recent invariant, the relative symplectic cohomology by Varolgunes \cite{Var}, to give a sufficient condition for some compact sets being heavy. As an application, we find many examples of singular Lagrangian sets that are heavy. 

The relative symplectic cohomology assigns a module $SH_{M}(K; \Lambda_{0})$ over the Novikov ring, and an algebra $SH_{M}(K; \Lambda)$ over the Novikov field to any compact subset $K$ of a closed symplectic manifold $M$. This assignment tells us how the quantum cohomology ring $QH(M)$ is distributed among compact subsets of $M$. And these pieces can be glued together by a notable Mayer-Vietoris process \cite{Var}.

This idea of distributing the quantum cohomology ring to compact sets was further explored by Dickstein-Ganor-Polterovich-Zapolsky \cite{DGPZ}. They successfully combined the theory of the relative symplectic cohomology with the theory of ideal-valued measures which was first studied by Gromov \cite{Gromov}, to construct a \textit{symplectic ideal-valued quasi-measure} on any closed symplectic manifold. In particular, a new notion of SH-heavy sets was defined. Comparing these two notions of heaviness becomes an interesting problem. They proposed the following conjecture.

\begin{conjecture}[Conjecture 1.52 \cite{DGPZ}]\label{con}
	A compact subset of a closed symplectic manifold is heavy if and only if it is SH-heavy.
\end{conjecture}

Under an \textit{index bounded} condition, Dickstein-Ganor-Polterovich-Zapolsky proved that heaviness implies SH-heaviness.

\begin{theorem}[Corollary 1.55 \cite{DGPZ}]\label{one}
	If $(M, \omega)$ is symplectically aspherical and $K$ is a heavy contact-type region with incompressible index bounded boundary, then $K$ is SH-heavy.
\end{theorem}

Our first result is to prove that SH-heaviness implies heaviness, under a similar geometric setting.

\begin{theorem}[Corollary \ref{c:main}]\label{heavy}
Let $(M, \omega)$ be a symplectically aspherical manifold and let $K$ be an index bounded domain. If $K$ is SH-heavy then $K$ is heavy.\footnote{In this article we only talk about heaviness with respect to the unit of the quantum cohomology.}
\end{theorem}

\begin{remark}
	We remark that our definition of the index bounded condition is slightly different than those in \cite{DGPZ} and \cite{TVar}. See (\ref{ib}).
\end{remark}

The idea of the proof is motivated by a condition that certain smooth Lagrangian submanifold is heavy, where the quantum cohomology and the Lagrangian Floer theory are related by a closed-open map. See \cite{EP} for the monotone case and \cite{FOOO} for the general case.

\begin{theorem}[Theorem 1.6 \cite{FOOO}]
	For a closed relative spin weakly unobstructed Lagrangian submanifold $L$ of a symplectic manifold $M$. If the self Lagrangian Floer cohomology $HF(L)$ is non-zero, then $L$ is heavy.
\end{theorem}

We will have an analogue of this theorem for the relative symplectic cohomology.

\begin{theorem}[Theorem \ref{t:mainproof}]\label{main}
	Let $(M, \omega)$ be a symplectically aspherical manifold and let $K$ be an index bounded domain. If $SH_M(K; \Lambda)$ is non-zero then $K$ is heavy.
\end{theorem}

The invariant $SH_M(K; \Lambda)$ is a unital $\Lambda$-algebra, where the product structure was constructed by Tonkonog-Varolgunes \cite{TVar}. Here we study the relation between its unit and certain spectral invariants. Then Theorem \ref{heavy} can be deduced from Theorem \ref{main}.

Now we discuss applications of the above theorem. The relative symplectic cohomology sometimes only depends on the intrinsic geometry of $K$, not on the embedding of $K$ into $M$. Consider an index bounded domain $K$ in a symplectic aspherical manifold $(M, \omega)$, then $(K, \omega\mid_{K})$ is a convex symplectic manifold. One can define the classical symplectic cohomology $SH(\hat{K}; \Lambda)$, see \cite{Vit} and Section 5.3.2 \cite{DGPZ}. Here $\hat{K}$ is the symplectic completion of $K$. 

\begin{theorem}[Theorem 1.57 \cite{DGPZ}]\label{t:SH}
	For $K$ and $M$ as above, if $\pi_{1}(\partial K)\to \pi_{1}(M)$ is injective, then $SH_M(K; \Lambda)$ is isomorphic to $SH(\hat{K}; \Lambda)$. Moreover
	$$
	\ker(H^*(M; \Lambda)\to H^*(K; \Lambda))\subset \ker(r:SH_M(M; \Lambda)\to SH_M(K; \Lambda)).
	$$
\end{theorem}

Hence Theorem \ref{main} and Theorem \ref{t:SH} together give many examples of heavy sets, by only considering intrinsic properties of $K$.

\begin{corollary}
	Let $K$ and $M$ be as above with $\pi_{1}(\partial K)\to \pi_{1}(M)$ being injective. If $SH(\hat{K}; \Lambda)\neq 0$ then $K$ is heavy in $M$.  
\end{corollary}
\begin{proof}
	This follows from Theorem \ref{main} and Theorem \ref{t:SH}. 
\end{proof}

\begin{corollary}\label{c:complement}
	Let $K$ and $M$ be as above with $\pi_{1}(\partial K)\to \pi_{1}(M)$ being injective. Then $\overline{M-K}$ is heavy in $M$. Particularly, $K$ is not super-heavy in $M$.
\end{corollary}
\begin{proof}
	Since $K$ is a compact domain in $M$, the top degree volume class of $M$ is always in $$
	\ker(H^*(M; \Lambda)\to H^*(K; \Lambda)).
	$$ 
	By Theorem \ref{t:SH} it is also in $$
	\ker(r:SH_M(M; \Lambda)\to SH_M(K; \Lambda)).
	$$ 
	Then using the Mayer-Vietoris property, we can show $SH_M(\overline{M-K}; \Lambda)\neq 0$. See Theorem \ref{t:comp} and Corollary \ref{c:comp} for more details, where we study the heaviness of the complement of an index bounded domain.
\end{proof}

A sample example of the above corollary is that

\begin{example}
	Let $M$ be a symplectically aspherical manifold with dimension $\dim_M\geq 4$ and $K$ be a Weinstein neighborhood of a Lagrangian sphere $S$, induced by the round metric. The boundary $\partial K$ is, after small perturbation, index bounded with respect to a non-degenerate contact form. On the other hand, the symplectic cohomology $SH(\hat{K}; \Lambda)$ is known to be non-zero and infinite-dimensional. Hence the above two Corollaries tell us that both $K$ and $\overline{M-K}$ are heavy. We remark that the heaviness of $K$ also follows that $S$ has a non-zero self Floer cohomology since it is weakly exact. The heaviness of $\overline{M-K}$ is new. See Section 5 \cite{S} for other examples of index bounded Weinstein neighborhoods of Lagrangians.
\end{example}

Next we discuss another family of examples, coming from skeleta with respect to a divisor. For a closed symplectically aspherical manifold $M$ and a chosen Giroux divisor $D$ (Definition \ref{divsor}), a suitable skeleton $L$ in $M-D$ was constructed in \cite{TVar}. Also see \cite{BSV} for the skeleton in the monotone case. Moreover they proved that $SH_{M}(\bar{U}; \Lambda)\neq 0$ for any neighborhood $U$ of $L$, see Theorem 1.24 \cite{TVar}. Hence we can prove the heaviness of the skeleton $L$ by shrinking $U$. 

\begin{theorem}[Corollary \ref{c:skeleta}]\label{ske}
	For $(M, \omega)$ being symplectically aspherical and $D$ being a Giroux divisor in $M$, the Lagrangian skeleton $L$ is a heavy set.
\end{theorem}

One simple example of the above theorem is $M$ being a symplectic 2-torus and $D$ being a finite collection of points. Then the skeleton $L$ is a heavy set. Particularly, for $D$ being one point and for a suitable choice of Liouville vector field on the complement of $D$, $L$ is the union of circles and arcs. In this case, \cite{Ish} and \cite{Mor} already proved that $L$ is not only heavy but also super-heavy. Theorem \ref{ske} gives more examples of heavy sets, possibly singular and in higher dimensions.

Now we give a quick sketch of the proof of Theorem \ref{main}. The relative symplectic cohomology of the domain $K$ can be computed by using a family $\lbrace G_{n}\rbrace$ of increasing Hamiltonian functions which converge to zero on $K$ and go to infinity outside $K$. Given any Hamiltonian function $f$ on $M$ which is slightly less than zero on $K$, see Figure \ref{Ham}, we use continuation maps between $CF(f)$ and $CF(G_{n})$ to estimate the spectral number of the unit with respect to $f$. By a special choice of $G_n$, the outputs of the continuation maps could be orbits in $K$, which we call ``lower orbits'', or could be outside $K$, which we call ``upper orbits''. One key technical procedure is a process of ``ignoring upper orbits'' (Lemma \ref{l:ig}) developed in \cite{M2020, S}. See also \cite{BSV} for a similar process in the monotone case, but with a different proof. By this process we know the output of the continuation map are lower orbits, whose Hamiltonian value is roughly zero. Together with the index bounded condition, we can control the action of the output of the continuation map. This gives the desired estimate of the spectral number of the unit.

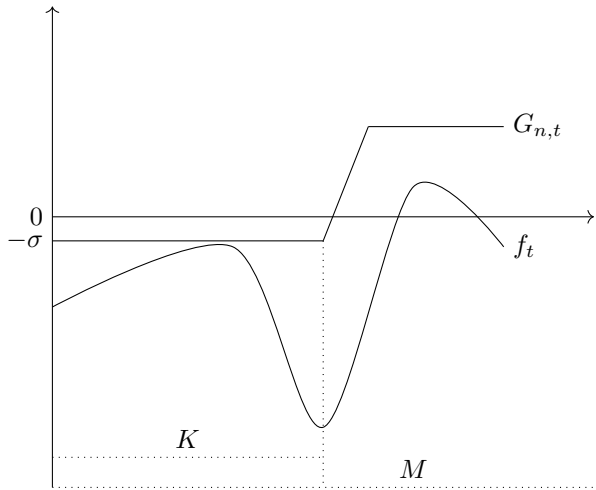
\begin{figure}
	\begin{tikzpicture}[xscale=1.2, yscale=0.8]
		\draw plot [smooth] coordinates {(0,1) (2,2) (3,-1) (4,3) (5,2)};
		\draw [->] (0,-2)--(0,6);
		\draw [->] (0,2.5)--(6,2.5);
		\draw (0,2.1)--(3,2.1);
		\draw (3,2.1)--(3.5,4);
		\draw (3.5,4)--(5,4);
		
		\draw node [right] at (5,4) {$G_{n,t}$};
		
		\draw node [right] at (5,2) {$f_t$};
		\draw node [left] at (0,2.1) {$ -\sigma$};
		\draw node [left] at (0,2.5) {$0$};
		\draw [dotted] (3,2.1)--(3,-2);
		\draw [dotted] (0,-1.5)--(3,-1.5);
		\draw node [above] at (1.5,-1.5) {$K$};
		\draw [dotted] (0,-2)--(6,-2);
		\draw node [above] at (4,-2) {$M$};
	\end{tikzpicture}
	\caption{Hamiltonian functions.}\label{Ham}
\end{figure}

\begin{remark}
The result ``$SH_M(K;\Lambda)\neq 0$ implies that $K$ is heavy'' is now proved in full generality in \cite{MSVar}. Their argument uses a chain level algebra structure of the relative symplectic cohomology which is developed in \cite{AGV}. Hence we feel it is still worth recording the proof of Theorem \ref{main} here, which only uses the non-Archimedean vector space structure under a more restricted geometric setting. The proofs in \cite{MSVar} and in this article are different and independent.

In the first version of this article in arXiv, some results were claimed in the symplectic Calabi-Yau setting without the aspherical condition. We withdraw those claims in this article while expecting they are true. The reasons are the following. First, its proof requires certain properties of one version of the relative symplectic cohomology whose generators are capped orbits rather than orbits. Establishing such a theory is expected but beyond the scope of the current article. Second, as mentioned above, general situations can be dealt with by using different methods. So we choose to focus on the symplectically aspherical case here, to illustrate the usage of the geometry of index bounded domains.
\end{remark}

\subsection*{Acknowledgments}
We thank Mark McLean for his generous guidance on related topics. We thank Leonid Polterovich for having arranged a meeting between Umut Varolgunes and the author, which started the author's discussions with Cheuk Yu Mak and Umut Varolgunes. We thank Adi Dickstein, Yaniv Ganor, Leonid Polterovich and Frol Zapolsky for helpful discussions. We thank an anonymous referee for very constructive input. The author was partially supported by the EPSRC grant EP/W015889/1.

\section{Floer theory background}
We work on a closed symplectically aspherical manifold $(M, \omega)$. This means that
$$
\omega\mid_{\pi_{2}(M)}= c_1(TM)\mid_{\pi_{2}(M)}=0.
$$
In this section, we briefly review the Hamiltonian Floer theory, spectral invariants and relative symplectic cohomology.

\subsection{Hamiltonian Floer theory}
We refer to \cite{HS} for more details on Hamiltonian Floer theory. However, our sign conventions are the same as in Section 3 \cite{Var}, which will be clarified later in the context.

Given a time-dependent Hamiltonian function $H_t: [0,1]\times M\to \mathbb{R}$, its Hamiltonian vector field $X_{H_t}$ is determined by $dH_t= \omega(X_{H_t}, \cdot)$. We say $H_t$ is non-degenerate if all its one-periodic Hamiltonian orbits are non-degenerate. In this article, we only study the contractible one-periodic orbits. The action of such an orbit $\gamma$ is
$$
\mathcal{A}_{H_t}(\gamma):= \int_{\gamma}H_t +\int_{w}\omega
$$
where $w: D^2\to M$ is a spanning disk of $\gamma$. An orbit $\gamma$ has a degree
$$
\abs{\gamma}:= n+ CZ(\gamma)
$$
where $CZ(\gamma)$ is the Conley-Zehnder index of $\gamma$ with respect to some spanning disk. Under the symplectic aspherical condition, the assignments of action and degree are independent of the spanning disk. If $H_t$ is a $C^2$-small Morse function and $\gamma$ is a critical point of $H_t$, then $\abs{\gamma}$ equals the Morse index of $\gamma$.

The Hamiltonian Floer chain group $CF^k(H_t)$ is the $\mathbb{C}$-vector space generated by degree-$k$ contractible one-periodic orbits of $H_t$. There is a Floer differential $d: CF^k(H_t)\to CF^{k+1}(H_t)$, which makes all $CF^k(H_t)$ into a chain complex. It is defined by counting Floer cylinders
$$
d(x):= \sum_{\abs{y}=\abs{x}+1} \sharp\mathcal{M}(x,y)y.
$$
We call this chain complex the Hamiltonian Floer complex, and call its homology the Hamiltonian Floer homology $HF^k(H_t)$. In this subsection, all complexes are over $\mathbb{C}$ and we omit it in the notation.

Given two non-degenerate Hamiltonians $H^0_t, H^1_t$ and a suitable homotopy $H^s_t$ of functions connecting them, there is a continuation map
$$
h^{01}: CF^k(H^0_t)\to CF^k(H^1_t),
$$
which is also defined by counting certain Floer cylinders
$$
h^{01}(x):= \sum_{\abs{y}=\abs{x}} \sharp\mathcal{M}(x,y)y.
$$
It is a chain map hence it induces a map, also written as $h^{01}$, from $HF^k(H^0_t)$ to $HF^k(H^1_t)$. When the $s$-derivative of $H^s_t$ is non-negative, we say it is a monotone homotopy.

\begin{remark}
	In the above definitions, we need to choose families of almost complex structures to achieve regularity of moduli spaces of Floer cylinders. We usually omit them in the notation and refer to \cite{HS} for more details.
\end{remark}

In our convention, the Floer differential increases the action and degree of an orbit. If the homotopy $H^s_t$ between two Hamiltonians is monotone, then the continuation map does not decrease the action.

Let $f$ be a Morse function on $M$ and let $CM^*(f)$ be the Morse complex of $M$ (graded as a cohomology theory). There is a PSS map \cite{PSS}
$$
PSS_{H_t}: CM^k(f) \to CF^k(H_t)
$$
defined by counting spiked disks. It induces an isomorphism between $HM^*(f)\cong H^*(M)$ and $HF^*(H_t)$, which we call the PSS isomorphism and also write it as $PSS_{H_t}$.

The PSS map is compatible with the continuation maps. Given two Hamiltonians $H^0_t, H^1_t$ and a homotopy between them, we have 
$$
h^{01}\circ PSS_{H^0_t}= PSS_{H^1_t}
$$
in the homology level.

Given a non-zero class $A\in H^k(M)$ and a Hamiltonian $H_t$, we define the spectral invariant 
$$
c(A, H_t):= \sup \{\mathcal{A}_{H_t}(\gamma)\mid \gamma\in CF^k(H_t), d\gamma =0, [\gamma]= PSS_{H_t}(A) \}.
$$
In most cases we care about the spectral invariant $c(1, H_t)$ of the unit $1\in H^0(M)$. These spectral invariants satisfy lots of good properties. For example, see subsection 3.4 in \cite{EP}. Particularly, by using approximation we can define spectral invariants for any smooth function on $M$. The homogenized spectral invariant of a time-independent function is defined as
$$
\mu(1, H):= \lim_{k\rightarrow \infty} \dfrac{c(1, kH)}{k}.
$$
Then we have the following definition of heavy sets.

\begin{definition}
	A compact subset $K$ of $M$ is called heavy if for any smooth function $H$ on $M$, we have $\mu(1, H)\leq \max_{K}H$.
\end{definition}

We have several remarks about this definition.

\begin{remark}\label{r:def}
	\begin{enumerate}
		\item Our sign conventions here make the Hamiltonian Floer theory a cohomological theory, while in \cite{EP} it is a homological theory. See Section 4.2 \cite{LZ} for a comparison. Hence our definition above is in terms of $\max$ rather than $\min$.
		\item Strictly speaking, the above definition should be ``heavy with respect to the unit $1$''. Generally one can also talk about heaviness with respect to other idempotent of the quantum cohomology. In this article we only talk about the unit.
		\item The spectral invariant has the shifting property: $\mu(1, H+C)= \mu(1, H)+ C$ for any constant $C$. Hence we only need to check functions which are non-positive on $K$ to verify the above definition.
	\end{enumerate}
\end{remark}

\subsection{Relative symplectic cohomology}

Now we review the relative symplectic cohomology by Varolgunes \cite{Var}.

The coefficient rings will be used are the Novikov ring
$$
\Lambda_{0}=\lbrace \sum_{i=0}^{\infty}a_{i}T^{\lambda_{i}}\mid a_{i}\in \mathbb{C}, \lambda_{i}\in\mathbb{R}_{\geq 0}, \lambda_{i}<\lambda_{i+1}, \lim_{i\rightarrow \infty}\lambda_{i}=+\infty \rbrace
$$
and the Novikov field
$$
\Lambda=\lbrace \sum_{i=0}^{\infty}a_{i}T^{\lambda_{i}}\mid a_{i}\in \mathbb{C}, \lambda_{i}\in\mathbb{R}, \lambda_{i}<\lambda_{i+1}, \lim_{i\rightarrow \infty}\lambda_{i}=+\infty \rbrace.
$$
Here $T$ is a formal variable of degree zero. A valuation defined on $\Lambda$ and $\Lambda_{0}$ is
$$
v(0\neq\sum_{i=0}^{\infty}a_{i}T^{\lambda_{i}}):= \min\{\lambda_{i}\mid a_i\neq 0\}, \quad v(0):= +\infty.
$$

Given a non-degenerate Hamiltonian $H_t$, let $CF^k(H_t; \Lambda_{0})$ be the free $\Lambda_{0}$-module generated by degree-$k$ contractible one-periodic orbits of $H_t$. We extend the valuation from $\Lambda_{0}$ to $CF^k(H_t; \Lambda_{0})$ as follows. For an element $x= \sum a_i \gamma_{n} \in CF^k(H_t; \Lambda_{0})$, define
$$
v(x):= \min_i \{ v(a_i)\}.
$$
That is, the valuation of an element in $CF^k(H_t; \Lambda_{0})$ is determined by its coefficients in $\Lambda_{0}$, independent of the orbits. The weighted Floer differential is defined as the classical Floer differential weighted by the action difference.
$$
d_T(x):= \sum_{\abs{y}=\abs{x}+1} \sharp\mathcal{M}(x,y)yT^{\mathcal{A}_{H_t}(y)- \mathcal{A}_{H_t}(x)}.
$$
We call this chain complex the weighted Hamiltonian Floer complex, and call its homology the weighted Hamiltonian Floer homology $HF^k(H_t; \Lambda_{0})$. Similarly, given two non-degenerate Hamiltonians $H^0_t, H^1_t$ and a monotone homotopy $H^s_t$ connecting them, there is a weighted continuation map
$$
h_T^{01}: CF^k(H^0_t; \Lambda_{0})\to CF^k(H^1_t; \Lambda_{0}),
$$
by defining
$$
h_T^{01}(x):= \sum_{\abs{y}=\abs{x}} \sharp\mathcal{M}(x,y)yT^{\mathcal{A}_{H^1_t}(y)- \mathcal{A}_{H^0_t}(x)}.
$$
Since our homotopy is monotone, the continuation map does not decrease the action so above maps are well-defined over $\Lambda_{0}$.

Now let $K$ be a compact subset of $M$. Consider a sequence of non-degenerate Hamiltonians $\{ H_{n,t}\}$ such that 
\begin{enumerate}
	\item $H_{n,t}\leq H_{n+1,t}$ for all $n\geq 1$.
	\item $H_{n,t}$ converge to zero on $K$ and diverge to positive infinity outside $K$.
\end{enumerate}
Then we choose suitable families of almost complex structures and monotone homotopies connecting adjacent $H_{n,t}$ and $H_{n+1,t}$, to get a sequence of Hamiltonian Floer complexes, connected by continuation maps
$$
\mathcal{C}:= CF^*(H_{1,t}; \Lambda_{0})\to CF^*(H_{2,t}; \Lambda_{0})\to \cdots.
$$
We call such a sequence of complexes a Floer one-ray.

Next we use a Floer one-ray to form a new complex, called the Floer telescope. Its underlying complex is defined as
$$
tel^*(\mathcal{C}):=\oplus_{n}(CF^*(H_{n,t}; \Lambda_{0})\oplus CF^*(H_{n,t}; \Lambda_{0})[1])
$$ 
where $CF^*(H_{n,t}; \Lambda_{0})[1]$ means shifting the degree by one. By using the weighted Floer differential $d_T^n$ and the weighted continuation map $h_T^{n(n+1)}$, we define the differential $\delta$ of the telescope as follows. If $x_{n}\in CF^k(H_{n}; \Lambda_{0})$ then
\begin{equation}
	\delta x_{n}= (-1)^{k}d_{T}^n x_{n} \in CF^{k+1}(H_{n}; \Lambda_{0}), 
\end{equation}
and if $x'_{n}\in CF^k(H_{n}; \Lambda_{0})[1]$ then
\begin{equation}\label{eq:delta}
	\begin{aligned}
		\delta x'_{n}&= ((-1)^{k}x'_{n}, (-1)^{k+1}d_{T}^n x'_{n}, (-1)^{k+1}h^{n(n+1)}_T x'_{n})\\
		&\in CF^{k}(H_{n}; \Lambda_{0})\oplus CF^{k+1}(H_{n}; \Lambda_{0})[1]\oplus CF^k(H_{n+1}; \Lambda_{0}).
	\end{aligned}
\end{equation}
One can check that $\delta^2=0$ hence we have a complex $(tel^*(\mathcal{C}), \delta)$. A typical element in $tel^*(\mathcal{C})$ will be written as
$$
x= (x_1, x_1', x_2, x_2', \cdots).
$$
Since $tel^*(\mathcal{C})$ is a direct sum, there are only finitely many non-zero terms in $x$. We define a valuation on $tel^*(\mathcal{C})$ as
$$
v(x):= \min_i \{ v(x_i), v(x_i')\},
$$
which extends the valuation on $CF^*(H_{n,t}; \Lambda_{0})$.

The completion of $tel^*(\mathcal{C})$ is defined as
$$
\widehat{tel^*(\mathcal{C})}:= \varprojlim_{r\to +\infty} tel^*(\mathcal{C})\otimes_{\Lambda_{0}} (\Lambda_{0}/T^r\cdot \Lambda_{0}).
$$
One can check that the differential $\delta$ extends to the completion. Moreover, one can show that $\widehat{tel^*(\mathcal{C})}$ has a more concrete expression, see 2.4 \cite{Var}
$$
\widehat{tel^*(\mathcal{C})}= \{ \sum_{l=1}^{+\infty} x^l\mid x^l\in tel^*(\mathcal{C}), \lim_{l}v(x^l)= +\infty \}.
$$
That is, an element 
$$
x= (x_1, x_1', x_2, x_2', \cdots)\in \Pi_{n}(CF^*(H_{n,t}; \Lambda_{0})\oplus CF^*(H_{n,t}; \Lambda_{0})[1])
$$
is in  $\widehat{tel^*(\mathcal{C})}$ if and only if $v(x_i), v(x_i')$ go to positive infinity.

\begin{definition}
	The homology of the completed telescope $(\widehat{tel^*(\mathcal{C})}, \delta)$ is called the relative symplectic cohomology of $K$ in $M$ over $\Lambda_{0}$, written as $SH_M(K; \Lambda_{0})$.
\end{definition}

In Proposition 3.3.4 \cite{Var}, it is shown that for different choices of almost complex structures, defining Hamiltonians and homotopies, the resulting homology groups of the completed telescopes are isomorphic. Hence $SH_M(K; \Lambda_{0})$ is an invariant of $K$ and $M$. This invariant has lots of good properties, notably the Mayer-Vietoris property. We list the properties that will be used in this article, and refer to \cite{Var} for others.

\begin{theorem}[Section 1 \cite{Var}]\label{t:property}
	The invariant $SH_M(K; \Lambda_{0})$ satisfies that
	\begin{enumerate}
		\item For compact sets $K_0\subset K_1$, there is a module map 
		$$
		r_{10}: SH_M(K_1; \Lambda_{0})\to SH_M(K_0; \Lambda_{0}),
		$$ 
		called the restriction map. If we have $K_0\subset K_1\subset K_2$, then $r_{10}\circ r_{21}= r_{20}$.
		\item $SH_M(M; \Lambda_{0})\cong H(M; \mathbb{C})\otimes_{\mathbb{C}}\Lambda_{+}$, where $\Lambda_{+}$ is the maximal ideal of $\Lambda_{0}$.
		\item Let $K_{1}, K_{2}$ be two compact domains with disjoint boundary. Then there is a long exact sequence
		$$
		\begin{aligned}
			\cdots&\to SH_{M}(K_{1}\cup K_{2}; \Lambda_{0})\to SH_{M}(K_{1}; \Lambda_{0})\oplus SH_{M}(K_{2}; \Lambda_{0})\\
			&\to SH_{M}(K_{1}\cap K_{2}; \Lambda_{0})\to\cdots.
		\end{aligned}
        $$ 
	\end{enumerate}
\end{theorem}

We will also use the relative symplectic cohomology over the Novikov field.

\begin{definition}
	The relative symplectic cohomology of $K$ in $M$ over $\Lambda$ is
	$$
	SH_M(K; \Lambda):= SH_M(K; \Lambda_{0})\otimes_{\Lambda_{0}} \Lambda.
	$$
\end{definition}

There is another definition of $SH_M(K; \Lambda)$ which will be frequently used, see Remark 2.4 \cite{TVar}. For a non-degenerate Hamiltonian $H_{t}$, we define $CF^k(H_{t}; \Lambda)$ to be the $\Lambda$-vector space generated by degree-$k$ contractible one-periodic orbits of $H_t$. The differential here is the weighted Floer differential. Then we use the above sequence of Hamiltonians $\{ H_{n,t}\}$ to form a Floer one-ray
$$
\mathcal{C}_\Lambda:= CF^*(H_{1,t}; \Lambda)\to CF^*(H_{2,t}; \Lambda)\to \cdots
$$
by using weighted continuation maps. The telescope $tel^*(\mathcal{C}_\Lambda)$ carries a valuation $v$ which is still defined by the minimum valuation as above. We can complete $tel^*(\mathcal{C}_\Lambda)$ with respect to this valuation to get $\widehat{tel^*(\mathcal{C}_\Lambda)}$. It can be shown that $H(\widehat{tel^*(\mathcal{C}_\Lambda)})$ is isomorphic to $SH_M(K; \Lambda)$. Similar to the above situation over $\Lambda_0$, the completed telescope $\widehat{tel^*(\mathcal{C}_\Lambda)}$ has a concrete expression
$$
\widehat{tel^*(\mathcal{C}_\Lambda)}= \{ \sum_{l=1}^{+\infty} x^l\mid x^l\in tel^*(\mathcal{C}_\Lambda), \lim_{l}v(x^l)= +\infty \}.
$$
That is, an element 
$$
x= (x_1, x_1', x_2, x_2', \cdots)\in \Pi_{n}(CF^*(H_{n,t}; \Lambda)\oplus CF^*(H_{n,t}; \Lambda)[1])
$$
is in  $\widehat{tel^*(\mathcal{C}_\Lambda)}$ if and only if $v(x_i), v(x_i')$ go to positive infinity.

In \cite{TVar}, it is shown that $SH_M(K; \Lambda)$ is a unital $\Lambda$-algebra and the restriction maps respect the units.

\begin{theorem}(Subsection 5.5 \cite{TVar})
	For each compact subset $K$ of $M$, there is an element $e_K\in SH^0_M(K; \Lambda)$ called the unit. It satisfies the following
	\begin{enumerate}
		\item $SH_M(K; \Lambda)=0$ if and only if $e_K=0$.
		\item For $K_0\subset K_1$, we have $r_{10}(e_{K_1})=e_{K_0}$.
	\end{enumerate}
\end{theorem}

\subsection{Index bounded domains}
We consider a triple $(K, \partial K, \alpha)$ such that
\begin{enumerate}
	\item $K$ is a compact domain in $M$.
	\item $\alpha$ is a contact form on $\partial K$.
	\item $d\alpha= \omega\mid_{\partial K}$.
	\item The local Liouville vector field points outwards along $\partial K$.
\end{enumerate}
Then $\partial K$ admits a neighborhood in $M$ that is symplectomorphic to $[1-\epsilon, 1+\epsilon]\times \partial K, d(r\alpha)$. Here $r$ is the coordinate on $[1-\epsilon, 1+\epsilon]$ and $\partial K$ is identified with $\{1\}\times \partial K$. Consider a Hamiltonian function $H$ on $M$ which equals $f(r)$ on $[1-\epsilon, 1+\epsilon]\times \partial K$. A standard computation shows that a one-periodic orbit $\gamma$ of $H$ which is contained in $[1-\epsilon, 1+\epsilon]\times \partial K$ is a Reeb orbit of the contact form $\alpha$. If $\alpha$ is a non-degenerate contact form and $\gamma$ is contractible in $M$, we have a well-defined Conley-Zehnder index $CZ(\gamma)$, with $\gamma$ viewed as a Hamiltonian orbit. Next we consider all the linear functions $H_{\lambda}=\lambda r$ for some positive slope $\lambda$. And for any integer $k$ we set 
$$
C_{\lambda, k}:=\sup_{\gamma}\lbrace \abs{\int_{w_\gamma} \omega}\rbrace.
$$
Here $\gamma$ runs over all one-periodic orbits of $H_{\lambda}$ which is contained in $[1-\epsilon, 1+\epsilon]\times \partial K$, contractible in $M$ and with $CZ(\gamma)=k$, and $w_{\gamma}$ is a spanning disk of $\gamma$. We say $(K, \partial K, \alpha)$ is \textit{index bounded} if $\alpha$ is non-degenerate and
\begin{equation}\label{ib}
	\sup_{\lambda \geq 0} \lbrace C_{\lambda, k} \rbrace< +\infty 
\end{equation}
for every integer $k$. This index bounded condition plays an important technical role in some recent studies of Hamiltonian Floer theory, starting from \cite{M2020}. A direct consequence is that for any given $k$, there is no index-$k$ one-periodic orbit of $f(r)$ in $[1-\epsilon, 1+\epsilon]\times \partial K$, if the slopes of $f(r)$ are large enough. 

\begin{remark}
	We formulate the above index bounded condition as a property of the embedding of $K$, rather than an intrinsic property of the contact manifold $(\partial K, \alpha)$. And it is defined by using Conley-Zehnder indices of Hamiltonian orbits rather than Reeb orbits.
	
	In many cases our definition is equivalent to Definition 1.53 \cite{DGPZ} and Definition 1.12 \cite{TVar}. For example, it is the case when $\pi_{1}(\partial K)\rightarrow \pi_{1}(M)$ is injective.
	
	To compare the Conley-Zehnder indices of Hamiltonian orbits with those of Reeb orbits, see Lemma 5.25 \cite{M2020}. They differ by a universally bounded amount.
\end{remark}

\subsection{Symplectic ideal-valued quasi-measures}
Next we review the quantum cohomology ideal-valued quasi-measures defined by Dickstein-Ganor-Polterovich-Zapolsky \cite{DGPZ}. 

Let $K$ be a compact set of a closed smooth symplectic manifold $M$. Define the quantum cohomology ideal-valued quasi-measure of $K$
\begin{equation}
	\tau(K):= \bigcap_{K\subset U} \ker(r: SH_{M}(M; \Lambda)\rightarrow SH_{M}(M-U; \Lambda))
\end{equation}
where $U$ runs over all open sets containing $K$. If $\tau(K)\neq 0$, then we say $K$ is SH-heavy. This measure $\tau$ satisfies several interesting properties \cite{DGPZ}, which indicate its importance in symplectic topology.

\section{Proofs}
In this section we prove Theorem \ref{main}. Let $(K, \partial K, \alpha)$ be an index bounded domain in a symplectically aspherical manifold $(M, \omega)$, we fix a collar neighborhood $U:=[1-\epsilon, 1+\epsilon]\times \partial K$ of $\partial K$. Recall that we assume $\alpha$ is a non-degenerate contact form, the set of periods of its Reeb orbits form a discrete subset of $\mathbb{R}_+$. We write this set as $Spec(\alpha)$.

\begin{definition}\label{d:ad}
	Given the triple $(K, \partial K, \alpha)$, a smooth function $H: M\to \mathbb{R}$ is called admissible if 
	\begin{enumerate}
		\item $H$ has small first and second derivatives outside $U$ such that it only has constant one-periodic orbits outside $U$ and they are non-degenerate.
		\item $H$ only depends on the radial coordinate $r$ in $U$. We write $H=f(r)$ on $U$.
		\item $H< 0$ in $K$ and $f(1+\epsilon/3)\geq 0$.
		\item $f(r)=\lambda r$ on $[1+\epsilon/3, 1+2\epsilon/3]\times \partial K$ for some positive $\lambda\notin Spec(\alpha)$.
		\item $f'(r)\notin Spec(\alpha)$ when $r\in[1,1+\epsilon/3]$.
		\item $f'(r)\geq 0$ on $[1-\epsilon, 1+\epsilon]\times \partial K$.
		\item $f''(r)\leq 0$ on $[1+\epsilon/3, 1+\epsilon]\times \partial K$.
		\item $f''(r)\geq 0$ on $[1-\epsilon, 1+\epsilon/3]\times \partial K$.
	\end{enumerate}
\end{definition}

\begin{figure}
	\begin{tikzpicture}[xscale=0.8, yscale=0.8]
		\draw [->] (-2,0.1)--(4.5,0.1);
		\node [right] at (4.5,0.1) {$r$};
		\draw (-1,-0.2) to [out=0, in=250] (1,1);
		\draw (1,1)--(2,4);
		\draw (2,4) to [out=70, in=180] (3, 5);
		\draw (-1,0)--(-1,0.2);
		\node [above] at (-1,0.3) {\small{$1-\epsilon$}};
		\draw (0,0)--(0,0.2);
		\node [above] at (0,0.3) {\small{$1$}};
		\draw (1,0)--(1,0.2);
		\node [below] at (1,-0.1) {\small{$1+\frac{\epsilon}{3}$}};
		\draw (2,0)--(2,0.2);
		\node [above] at (2,0.3) {\small{$1+\frac{2\epsilon}{3}$}};
		\draw (3,0)--(3,0.2);
		\node [below] at (3,-0.1) {\small{$1+\epsilon$}};
		\draw [snake it] (-2,-0.2)--(-1,-0.2);
		\draw [snake it] (3,5)--(4,5);
	\end{tikzpicture}
	\caption{Admissible Hamiltonian in the radial coordinate.}\label{admissible}
\end{figure}
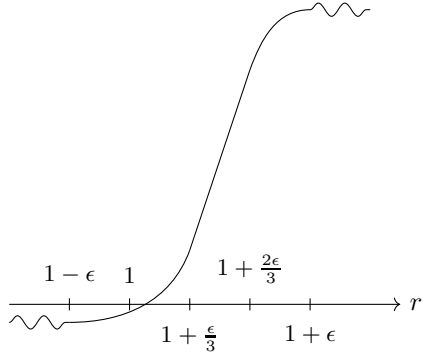

A depiction of admissible Hamiltonian is in Figure \ref{admissible}. The one-periodic orbits of an admissible $H$ fall into two groups: constant orbits outside $U$ and non-constant orbits in $U$. Since $H=f(r)$ in $U$, the non-constant orbits are multiples of Reeb orbits of $\alpha$. For a one-periodic orbit $\gamma$ of $H$, constant or non-constant, the function $H$ is constant on $\gamma$. We call this constant the Hamiltonian value of $\gamma$.

\begin{definition}\label{d:lower}
	A one-periodic orbit $\gamma$ of an admissible $H$ is called a lower orbit if its Hamiltonian value is less than zero. Otherwise it is called an upper orbit.
\end{definition}

By $(4)$ in Definition \ref{d:ad}, there is no one-periodic orbits in $[1+\epsilon/3, 1+2\epsilon/3]\times \partial K$. Hence the Hamiltonian value of an upper orbit of $H$ is at least $\lambda\epsilon/3$.

An admissible function $H$ is time-independent. Each of its non-constant one-periodic orbit carries an $S^1$-symmetry. We use the standard perturbation in \cite{CFHW} to break this symmetry which makes $H$ into a non-degenerate function $H_t$. The perturbations are supported in neighborhoods of non-constant orbits of $H$. They can be chosen to be arbitrarily small. We assume our perturbation satisfies that
\begin{enumerate}
	\item A non-constant lower orbit $\gamma$ of $H$ becomes two non-degenerate orbits $\gamma_{\pm}$, with $\int_{\gamma_{\pm}}H_t <0$.
	\item An non-constant upper orbit $\beta$ of $H$ becomes two non-degenerate orbits $\beta_{\pm}$, with $\int_{\beta_{\pm}}H_t >\lambda\epsilon/3$.
\end{enumerate}

\begin{definition}\label{d:tad}
	A time-dependent admissible function $H_t$ is one obtained from an admissible function $H$ by using above perturbations.
\end{definition}

For a time-dependent admissible function, it still makes sense to talk about its lower and upper orbits by above estimates.

One can directly check the existence of a sequence of time-dependent admissible Hamiltonians $\{ H_{n,t}\}$ such that 
\begin{enumerate}
	\item $H_{n,t}\leq H_{n+1,t}$ for all $n\geq 1$.
	\item $H_{n,t}$ converge to zero on $K$ and diverge to positive infinity outside $K$.
\end{enumerate}
Hence we can use $\{ H_{n,t}\}$ to compute $SH_M(K; \Lambda)$. In the following, when we write $CF(H_{n,t})$ we mean the Floer complex over $\mathbb{C}$ with the classical differential $d$, and when we write $CF(H_{n,t}; \Lambda)$ we mean the Floer complex over $\Lambda$ with the weighted differential $d_T$. Continuation maps should be understood in the same way. For any orbit $\gamma\in CF(H_{n,t})$, define 
\begin{equation}\label{e:change}
	J: CF(H_{n,t})\otimes_{\mathbb{C}}\Lambda \to CF(H_{n,t}; \Lambda), \quad J(\gamma\otimes a):= aT^{\mathcal{A}_{H_{n,t}}(\gamma)}\gamma.
\end{equation}
It is a chain isomorphism. Particularly, if $x\in CF(H_{n,t})$ is $d$-exact, then $x$, viewed as an element in $CF(H_{n,t}; \Lambda)$, is $d_T$-exact.

By the index bounded condition (\ref{ib}), there is a constant $C>0$ such that
$$
\sup_{n} \lbrace \abs{\int_{w_\gamma}\omega} \mid \gamma\in CF^*(H_{n,t}), *=0,1 \rbrace< C.
$$
Let $\lambda_{n}$ be the slope of $H_{n,t}$ in $[1+\epsilon/3, 1+2\epsilon/3]\times \partial K$, then we have that
\begin{lemma}\label{l:upper}
	There exists $N>0$ such that if $n\geq N$ then the action of a degree zero/one upper orbit of $H_{n,t}$ is strictly larger than the action of any degree zero/one lower orbit of $H_{k,t}$ for any $k$.
\end{lemma}
\begin{proof}
	This follows from $\lim_{n} \lambda_{n}= +\infty$, and two estimates
	\begin{enumerate}
		\item If $\gamma$ is a lower orbit of $H_{k,t}$, then $\mathcal{A}_{H_{k,t}}(\gamma)< C$ for any $k$.
		\item If $\gamma$ is an upper orbit of $H_{n,t}$, then $\mathcal{A}_{H_{n,t}}(\gamma)> \lambda_{n}\epsilon/3 -C$.
	\end{enumerate}
\end{proof}

Choosing monotone homotopies connecting $H_{n,t}, H_{n+1,t}$, we can consider the Floer one-ray
$$
\mathcal{C}_\Lambda= CF(H_{N,t}; \Lambda)\to CF(H_{N+1, t}; \Lambda)\to \cdots.
$$
We write the continuation maps as
$$
h^{k(k+1)}_T: CF^0(H_{N+k,t}; \Lambda)\to CF^0(H_{N+k+1,t}; \Lambda).
$$
Since $SH_M(K; \Lambda)$ is independent of the choice of defining Hamiltonians, we have $H(\widehat{tel(\mathcal{C}_\Lambda)})\cong SH_M(K; \Lambda)$.

Pick a closed element $x\in CF^0(H_{N+k,t}; \Lambda)$. It gives an element
$$
\tilde{x}=(0,0\cdots, x,0, \cdots)\in \widehat{tel(\mathcal{C}_\Lambda)}.
$$
That is, the element with $x$ in the spot of $CF^0(H_{N+k,t}; \Lambda)$ and with zeroes in other spots. One can directly check it is a closed element in the completed telescope. Moreover, we have the following.

\begin{lemma}\label{l:ig}
	If $x$ is an upper orbit of $H_{N+k,t}$, then $\tilde{x}$ is exact in $\widehat{tel(\mathcal{C}_\Lambda)}$.
\end{lemma}
\begin{proof}
	By Lemma \ref{l:upper}, $h^{k(k+1)}_T(x)$ can only be a linear combination of upper orbits of $CF^0(H_{N+k+1,t}; \Lambda)$ since continuation maps do not decrease the action.
	
	Then consider the element 
	$$
	\hat{x}:= (0,0, \cdots, 0, x, 0, h^{k(k+1)}_T(x), 0, h^{(k+1)(k+2)}_T\circ h^{k(k+1)}_T(x), \cdots).
	$$
	More precisely, on the spot of $CF^0(H_{N+k,t}; \Lambda)[1]$ it is $x$, and on the spot of $CF^0(H_{N+k+l,t}; \Lambda)[1]$ it is $h^{(k+l-1)(k+l)}_T\circ\cdots\circ h^{k(k+1)}_T(x)$ for $l\geq 1$. On other spots it is zero. By estimate $(2)$ in Lemma \ref{l:upper}, the valuation of $h^{(k+l-1)(k+l)}_T\circ\cdots\circ h^{k(k+1)}_T(x)$ goes to infinity as $l$ goes to infinity, hence $\hat{x}$ is a well-defined element in $\widehat{tel(\mathcal{C}_\Lambda)}$. The definition of the telescope differential $\delta$ shows that $\delta(-\hat{x})= \tilde{x}$. Moreover, the valuation of $-\hat{x}$ is the same as $\tilde{x}$.
\end{proof}

The above Lemma is called the process of \textit{ignoring upper orbits}. Next we discuss how heaviness can be studied by the above telescope.

Suppose that $K$ is not heavy, then there is a function $\tilde{f}$ such that $\tilde{f}\leq 0$ on $K$ and $\mu(1, \tilde{f})>0$, see Remark \ref{r:def} $(3)$. Then we have some $k>0$ such that $c(1, k\tilde{f})> 2(1+C)$ where $C$ is the constant in Lemma \ref{l:upper}. Pick a small positive number $\sigma$ and define $f:=k\tilde{f}- 2\sigma$ such that $f< -\sigma$ on $K$ and $c(1, f)> 1+C$. We perturb $f$ to get a non-degenerate $f_t$ which also satisfies that $f_t< -\sigma$ on $K$ and $c(1, f_t)> 1+C$.

The functions $H_{N+k, t}$ converge to zero on $K$ and diverge to positive infinity outside $K$. There is some $N'$ such that $H_{N+N', t}\geq f_t$. We define $G_{n,t}:= H_{N+N'+n, t}$ for notation simplicity, see Figure \ref{Ham}. Then the completion of the telescope
$$
\mathcal{G}: CF(G_{1,t}; \Lambda)\to CF(G_{2,t}; \Lambda)\to \cdots 
$$
computes $SH_M(K; \Lambda)$. On the other hand, we have the following telescope
$$
\mathcal{I}_f: CF(f_t; \Lambda)\to CF(f_t; \Lambda)\to \cdots
$$
where the continuation maps are identities. Pick a closed element $x\in CF^0(f_t; \Lambda)$ and a sequence $\{a_n\}\in \Lambda, v(a_n)\to +\infty$. We have two well-defined closed elements in the complete telescope
$$
(\sum_n a_n x, 0,0,0,\cdots), \quad (a_1 x,0,a_2 x,0,\cdots).
$$
By using the telescope differential, we can check they are homologous. Consider 
$$
\hat{x}:=(0,\sum_{n\geq 2} a_n x,0,\sum_{n\geq 3} a_n x,0,\cdots)
$$
which is $\sum_{n\geq k+1} a_n x$ on the spot of the $k$-th $CF^0(f_t; \Lambda)[1]$ and is zero otherwise. Then we have
$$
\delta(\hat{x})= (\sum_n a_n x, 0,0,0,\cdots)- (a_1 x,0,a_2 x,0,\cdots).
$$

Next we discuss the relation between PSS maps and the restriction map. Pick a negative Morse function $g$ on $M$ satisfying that
\begin{enumerate}
	\item It has small first and second derivative such that, viewed as a Hamiltonian, it has only constant one-periodic orbits.
	\item It has a unique index zero critical point $\tau$. 
\end{enumerate}
Then consider $g_n:= a_n g+ b_n$ for some $a_n>0, b_n\in\mathbb{R}$. If we choose $a_n, b_n$ carefully we can achieve that
\begin{enumerate}
	\item $g_n$ satisfies the above two properties for $g$.
	\item $g_{n}\leq g_{n+1}$ for all $n$.
	\item $g_n\leq G_{n,t}$ for all $n$.
	\item $g_n$ converges to zero on $M$.
\end{enumerate}
We write the Morse complex of $g_n$ as $CM(g_n)$, and the Hamiltonian complex of $g_n$ as $CF(g_n), CF(g_n; \Lambda)$. The unique index zero critical point of $g_n$ is written as $\tau_n$. Hence $\tau_n$ is a closed element in $CM^0(g_n)$ and represents $1\in H(M)$. Pick linear monotone homotopies connecting $g_n, g_{n+1}$ and write
$$
\mathcal{I}_g: CF(g_1; \Lambda)\to CF(g_2; \Lambda)\to \cdots
$$
as the induced telescope. Similar to the case of $\mathcal{I}_f$, pick a sequence $\{a_n\}\in \Lambda, v(a_n)\to +\infty$, the element 
$$
(a_1\tau_1,0,a_2\tau_2,0,\cdots)\in\widehat{tel(\mathcal{I}_g)}
$$
is homologous to an element $(b\tau_1,0,0,0,\cdots)$ for some $b\in\Lambda$. This is because each $CF^0(g_n; \Lambda)$ has a unique generator, the continuation maps are identity maps weighted by the action differences.

Then we have three collections of maps: PSS maps
$$
PSS^{gf}_n: CM^0(g_n)\to CF^0(f_t), \quad PSS_{g_n}: CM^0(g_n)\to CF^0(g_n)
$$
and the continuation map
$$
h^{gf,n}: CF^0(g_n)\to CF^0(f_t).
$$
By the compatibility of PSS maps and continuation maps, $h^{gf,n}\circ PSS_{g_n}$ equals $PSS^{gf}_n$ in the homology level. Moreover, since $g_n$ has a unique index zero critical point, the map $PSS_{g_n}$ is identity in the chain level. So we get
\begin{lemma}\label{l:ho}
	$h^{gf,n}(\tau_n)$ is homologous to $PSS^{gf}_n(\tau_n)$ in $CF^0(f_t)$.
\end{lemma}

Recall that all $g_n$'s are negative and converge to zero on $M$, we have $H(\widehat{tel(\mathcal{I}_g)}) = SH_M(M; \Lambda)$. The collection of weighted continuation maps
$$
h^{gG,n}_T: CF(g_n; \Lambda)\to CF(G_{n,t}; \Lambda),
$$
together with suitable homotopies, induces a chain map 
$$
h^{gG}_T: \widehat{tel(\mathcal{I}_g)}\to \widehat{tel(\mathcal{G})}.
$$
In the homology level, it is the restriction map $r: SH_M(M; \Lambda)\to SH_M(K; \Lambda)$. Similarly, we have two other continuation maps
$$
h^{gf,n}_T: CF^0(g_n; \Lambda)\to CF^0(f_t; \Lambda), \quad h^{fG,n}_T: CF^0(f_t; \Lambda)\to CF^0(G_{n,t}; \Lambda)
$$
which also induce chain maps between corresponding completed telescopes.

\begin{lemma}\label{l:gluing}
	The three maps $h^{gG,n}_T, h^{gf,n}_T,h^{fG,n}_T$ induce well-defined chain maps $h^{gG}_T, h^{gf}_T,h^{fG}_T$ between completed telescopes. In the homology level we have that
	$$
	h^{gG}_T= h^{fG}_T\circ h^{gf}_T.
	$$
\end{lemma}
\begin{proof}
	The maps $h^{gG,n}_T,h^{fG,n}_T$ do not decrease valuation since $g_n\leq G_{n,t}, f_t\leq G_{n,t}$. The map $h^{gf,n}_T$ possibly decreases the valuation by a universal bounded amount since $f_t$ does not depend on $n$ and $g_n$ is uniformly bounded in $n$. Hence they all induce well-defined maps between corresponding completed telescopes. They are chain maps and $h^{gG}_T= h^{fG}_T\circ h^{gf}_T$ in the homology level follows from a gluing argument, similar to the proof of the functoriality of the restriction maps, see $(1)$ in Theorem \ref{t:property}.
\end{proof}

A direct corollary of this lemma is that if $h^{gG}_T=r\neq 0$ then $h^{fG}_T\neq 0$ in degree zero in the homology level. Now we use it to prove Theorem \ref{main}.

\begin{theorem}\label{t:mainproof}
	Let $(M, \omega)$ be a symplectically aspherical manifold and let $K$ be an index bounded domain. If $SH_M(K; \Lambda)$ is non-zero then $K$ is heavy.
\end{theorem}
\begin{proof}
	Suppose that $K$ is not heavy, we have the above functions $f_t, g_n, G_{n,t}$. 
	
	Let $y=(y_1,y_1',y_2,y_2',\cdots)$ be a closed degree zero element in $\widehat{tel(\mathcal{I}_g)}$, then it is of the form
	$$
	(a_1\tau_1,0,a_2\tau_2,0,\cdots)\in\widehat{tel(\mathcal{I}_g)}
	$$
	with $v(a_n)$ going to infinity. By the above discussion, it is homologous to $x=(b\tau_1,0,0,0,\cdots)$ for some $b\in\Lambda$.
	
	Since $c(1, f_t)>1+C$, there is a closed element $\gamma\in CF^0(f_t)$ which represents $PSS^{gf}_1(\tau_1)$ with action greater than $1+C$. By Lemma \ref{l:ho}, we have that $\gamma$ is homologous to $h^{gf,1}(\tau_1)$ in $CF^0(f_t)$. Hence $J(\gamma)$ is homologous to $J\circ h^{gf,1}(\tau_1)=h^{gf,1}_T\circ J(\tau_1)$ in $CF^0(f_t; \Lambda)$, see (\ref{e:change}). This implies that
	$$
	(J(\gamma),0,0,0,\cdots) \quad \text{and} \quad (h^{gf,1}_T (J(\tau_1)),0,0,0,\cdots)= (T^{\mathcal{A}_{g_1}(\tau_1)}h^{gf,1}_T(\tau_1),0,0,0,\cdots)
	$$
	are homologous in $\widehat{tel(\mathcal{I}_f)}$. On the other hand, 
	$$
	h^{gf}_T(x)= (bh^{gf,1}_T(\tau_1),0,0,0,\cdots).
	$$
	Hence $b^{-1}T^{\mathcal{A}_{g_1}(\tau_1)}h^{gf}_T(x)$ is homologous to $(J(\gamma),0,0,0,\cdots)$ in $\widehat{tel(\mathcal{I}_f)}$.

	Note that 
	$$
	h^{fG}_T((J(\gamma),0,0,0,\cdots))=(h^{fG,1}_T(J(\gamma)),0,0,0,\cdots).
	$$
	However, since the action of $\gamma$ is larger than $1+C$ and $h^{fG,1}$ does not decrease valuation, the geometric underlying orbits of $h^{fG,1}_T(\gamma)$ is a linear combination of upper orbits of $G_{1,t}$. By Lemma \ref{l:ig}, it is exact in $\widehat{tel(\mathcal{G})}$. Finally we get the following relations in the chain level
	$$
	\begin{aligned}
		r(y)&\sim r(x)\\
		&\sim h^{fG}_T\circ h^{gf}_T(x)\\
		&\sim h^{fG}_T(bT^{-\mathcal{A}_{g_1}(\tau_1)}((J(\gamma),0,0,0,\cdots)))\\
		&= bT^{-\mathcal{A}_{g_1}(\tau_1)}(h^{fG,1}_T(J(\gamma)),0,0,0,\cdots)
	\end{aligned}
	$$
	where $\sim$ means the homologous relation. It shows that for any closed degree zero element $y$ in $\widehat{tel(\mathcal{I}_g)}$, its image $r(y)$ under the restriction map is homologous to an exact element in $\widehat{tel(\mathcal{G})}$. Therefore $r$ is zero in the homology level, a contradiction to $SH_M(K; \Lambda)\neq 0$.

\end{proof}

\section{Applications}

First, we have an application on the heaviness of the complement of certain index bounded domains. Let $(M, \omega)$ be a symplectically aspherical manifold and let $K$ be an index bounded domain. Define $N:=\overline{M-K}$, which is also a compact domain with boundary $\partial N=\partial K$. In this case the local Liouville vector points inward along $\partial N$.

Fix a neighborhood of $\partial N$ in $M$ that is symplectomorphic to $([1-\delta, 1+\delta]\times \partial N, d(r\alpha))$. Here $r$ is the coordinate on $[1-\delta, 1+\delta]$, the vector field $\partial_r$ points inward along $\partial N$, and $\partial N$ is identified with $\{1\}\times \partial N$. Consider a Hamiltonian function $H_{-\lambda}$ on $M$ which equals $f(r)=-\lambda r$ on $[1-\delta, 1+\delta]\times \partial N$, with some negative slope $-\lambda$. For any integer $k$ we set 
$$
C_{\lambda, k}':=\sup_{\gamma'}\lbrace \abs{\int_{w_\gamma'} \omega}\rbrace.
$$
Here $\gamma'$ runs over all one-periodic orbits of $H_{-\lambda}$ which is contained in $[1-\delta, 1+\delta]\times \partial N$, contractible in $M$ and with $CZ(\gamma')=k$, and $w_{\gamma'}$ is a spanning disk of $\gamma'$. Then we observe that

\begin{lemma}\label{l:comp}
	If $(K, \partial K, \alpha)$ is an index bounded domain and $\pi_{1}(\partial K)\to \pi_{1}(M)$ is injective, then 
	$$
	\sup_{\lambda \geq 0} \lbrace C_{\lambda, k}' \rbrace< +\infty
	$$
	for every integer $k$.
\end{lemma}
\begin{proof}
	Since $\pi_{1}(\partial K)\to \pi_{1}(M)$ is injective, for any $\gamma'$ as above we can compute $\int_{w_\gamma'} \omega$ and the index of $\gamma'$ with respect to a spanning disk in $\partial K$. Any such $\gamma'$ is the reverse of some one-periodic orbit $\gamma$ of $H_{\lambda}=\lambda r, \lambda>0$. The indices of $\gamma, \gamma'$ are related by a minus sign plus a universal bounded error. Hence the estimate of $C_{\lambda, k}'$ follows from the index bounded condition (\ref{ib}) for $C_{\lambda, k}$.
\end{proof}

Then we can repeat all the arguments in the previous section for $N$.

\begin{theorem}\label{t:comp}
	Let $(M, \omega)$ be a symplectically aspherical manifold and let $K$ be an index bounded domain, with $\pi_{1}(\partial K)\to \pi_{1}(M)$ being injective. If $SH_M(N; \Lambda)\neq 0$ then $N$ is heavy.
\end{theorem}
\begin{proof}
	Similar to $G_{n,t}$ in Theorem \ref{t:mainproof}, we construct Hamiltonians $G_{n,t}'$ to compute $SH_M(K; \Lambda)$. They will have negative slopes with respect to $r$ in the neck region. They satisfy an index-action relation as in Lemma \ref{l:comp}. Then Lemma \ref{l:upper} and Lemma \ref{l:ig} work in the same way. The construction of $g_n$ is unchanged. 
\end{proof}

The following interpolation theorem \cite{M2020, TVar} between index bounded domains brings us more applications. 

\begin{theorem}
	[Proposition 1.13 \cite{TVar}] Let $M$ be a closed symplectic manifold with $c_1(TM)=0$, and $K$ be an index bounded domain. For a neck region $\partial K\times [1-\delta, 1+\delta]$ induced by the Liouville flow of the index bounded contact form, the restriction map
	$$
	r: SH_{M}(K\cup (\partial K\times [1-\delta, 1+\delta]); \Lambda)\rightarrow SH_{M}(K; \Lambda)
	$$
	is an isomorphism whenever $\delta >0$ is defined.
\end{theorem}

This interpolation theorem allows us to relate the SH-heaviness with heaviness.

\begin{lemma}\label{l:SH-heavy}
	Let $(M, \omega)$ be a closed symplectic manifold with $c_1(TM)=0$ and $K$ being an index bounded domain, if $K$ is SH-heavy then $SH_{M}(K)\neq 0$.
\end{lemma}
\begin{proof}
	Define $K_{1+\delta}:= K\cup (\partial K\times [1-\delta, 1+\delta])$.	The tubular neighborhood $[1-\delta, 1+\delta]\times \partial K$ of $\partial K$ gives a sequence of open sets $K^{\mathrm{o}}_{1+\delta}$ as the interior of $K_{1+\delta}$, parameterized by $\delta$. Then the quantum measure $\tau(K)$ can be computed as
	$$
	\tau(K)= \bigcap_{\delta >0}\ker(r: QH(M; \Lambda)\rightarrow SH_{M}(M-K^{\mathrm{o}}_{1+\delta}; \Lambda)).
	$$
	
	Suppose that $K$ is SH-heavy, which means that $\tau(K)\neq 0$, we will show that $SH_{M}(K; \Lambda)\neq 0$. By the interpolation theorem, we know that 
	$$
	SH_{M}(K; \Lambda)\cong SH_{M}(K_{1+\delta'}; \Lambda), \quad \forall -\delta < \delta'< \delta.
	$$
	Pick $0< \delta_{1} <\delta_{2}<\delta$ and suppose that, on the contrary, $SH_{M}(K; \Lambda)=0$. Then we have that $SH_{M}(K_{1+\delta_{2}}; \Lambda)=0$. By the Mayer-Vietoris property for relative symplectic cohomology, we have that $r:QH(M; \Lambda)\to SH_{M}(M-K^{\mathrm{o}}_{1+\delta_{1}}; \Lambda)$ is an isomorphism, which means that $\tau(K)=0$, contradiction.
\end{proof}

Combining this lemma with Theorem \ref{t:mainproof} we get the following corollary.

\begin{corollary}\label{c:main}
	For $(M, \omega)$ being symplectically aspherical and $K$ being an index bounded domain, if $K$ is SH-heavy then $K$ is heavy.
\end{corollary}
\begin{proof}
	If $K$ is SH-heavy then $SH_{M}(K; \Lambda)\neq 0$ by the above lemma. Applying Theorem \ref{t:mainproof} we complete the proof.
\end{proof}

Similar to the proof in Lemma \ref{l:SH-heavy}, we can prove that

\begin{corollary}\label{c:comp}
	For $(M, \omega)$ being symplectically aspherical and $K$ being an index bounded domain with $\pi_{1}(\partial K)\to \pi_{1}(M)$ being injective, then $\overline{M-K}$ is heavy. Particularly, $K$ is not super-heavy.
\end{corollary}
\begin{proof}
	We use the notation in Lemma \ref{l:SH-heavy}. Apply Theorem \ref{t:SH} to $K_{1+\delta}$, the volume class of $M$ is in 
	$$
	\ker(H(M; \Lambda)\to H(K_{1+\delta}; \Lambda))\subset \ker(r:SH_M(M; \Lambda)\to SH_M(K_{1+\delta}; \Lambda)).
	$$
	Then apply the Mayer-Vietoris sequence to the pair of $K_{1+\delta}$ and $\overline{M-K}$, we get that $SH_M(\overline{M-K}; \Lambda)\neq 0$. 
	
	By the interpolation invariance, we can also show $\overline{M-K_{1+\delta}}$ is heavy. This shows that $K$ is not super-heavy, since a heavy set intersects any super-heavy set.
\end{proof}

Next we move to a family of examples, which are the skeleta of symplectically aspherical manifolds, relative to simple crossings symplectic divisors. We refer the readers to \cite{FTMZ, M2020} for details on the theory of simple crossings symplectic divisors.

\begin{definition}
	[Definition 2.1 \cite{FTMZ}] Let $(M, \omega)$ be a closed symplectic manifold. A simple crossings symplectic divisor in $(M, \omega)$ is a finite transverse collection of $\lbrace V_{i}\rbrace_{i\in S}$ of closed submanifolds of $M$ of codimension 2, such that $V_{I}$ is a symplectic submanifold of $(M, \omega)$ for any $I\subset S$ and the intersection and $\omega$-orientations of $V_{I}$ agree.
\end{definition}

\begin{definition}[Definition 1.19 \cite{TVar}]\label{divsor}
	 A Giroux divisor $V=\cup_{i\in S} V_{i}$ is a simple crossings symplectic divisor in $(M, \omega)$ such that there exist integers $w_{i}>0$, a real number $c>0$ and
	$$
	\sum_{i} w_{i}PD[V_{i}] =c[\omega] \in H^{2}(M).
	$$
\end{definition}

Below is a structural result about complements of Giroux divisors.

\begin{proposition}[Proposition 1.20 \cite{TVar}]\label{str}
	 Let $V$ be a Giroux divisor in a symplectic manifold $(M, \omega)$ with $c_1(TM)=0$. Then there exists a Liouville domain $W\subset M-V$ such that
	\begin{enumerate}
		\item The closure of $M-W$ is stably displaceable.
		\item The closure of $M-W$ deformation retracts to $V$.
		\item $W$ is an index bounded domain.
	\end{enumerate}
\end{proposition}

\begin{definition}
	Let $V$ be a Giroux divisor in a closed symplectically aspherical manifold $(M, \omega)$. Given a Liouville domain $W$ satisfying Proposition \ref{str}. The image $W_{t}$ of $W$ under the time-$t$ reverse Liouville flow is a Liouville subdomain of $W$. The skeleton $L_{W}$ of $W$ is defined as $L_{W}=\cap_{t\in \mathbb{R}_{\geq 0}}W_{t}$. It is a compact subset of $M$ where the symplectic form vanishes.
\end{definition}

As the definition indicates, for a given Giroux divisor, there may be different Liouville domains satisfying Proposition \ref{str}. Hence we might have skeleta which are set-theoretically different. But they all have some symplectic rigidity properties.

\begin{theorem}
	[Theorem 1.24 \cite{TVar}] For any $t\in \mathbb{R}_{\geq 0}$, there is an isomorphism
	$$
	QH(M; \Lambda)\rightarrow SH_{M}(W_{t}; \Lambda).
	$$
\end{theorem}

Note that $W_t$ is an index bounded domain since it is the image of $W$ under the reverse Liouville flow. Then this theorem combined with Theorem \ref{t:mainproof} gives that
\begin{corollary}
	All $W_{t}$'s are heavy.
\end{corollary}

\begin{remark}
	The language of stable stems tells that $W_{0}=W$ is heavy since its complement is stably displaceable, see Subsection 1.2 \cite{EP}. Here the interpolation theorem actually shows that all $W_{t}$'s are heavy, which displays the power of the relative symplectic cohomology.
\end{remark}

\begin{corollary}\label{c:skeleta}
	The skeleton $L_{W}$ is heavy.
\end{corollary}
\begin{proof}
	Suppose that $L_{W}$ is not heavy, then there is a Hamiltonian function $F$ on $M$ with $\mu(1,F)> \max_{L_{W}}F$. Since $L_{W}$ is compact, the sets $W_{t}$ converge to $L_{W}$ in a uniform $C^{0}$-sense, and there exists $t>0$ such that
	$$
	\mu(1,F)> \max_{W_{t}}F\geq \max_{L_{W}}F
	$$
	which contradicts that $W_{t}$ is heavy.
\end{proof}

\bibliographystyle{amsplain}

\end{document}